\RequirePackage{fix-cm}

\documentclass[smallextended]{svjour3}       % onecolumn (second format)

\smartqed 

\usepackage{graphicx}
\usepackage{amsmath, enumitem}
\usepackage{lipsum}
\usepackage{amsfonts}
\usepackage{graphicx}
\usepackage{epstopdf}
\usepackage{algorithmic}
\usepackage{amsopn}
\usepackage{color,soul}
\usepackage{todonotes}

\ifpdf

\DeclareGraphicsExtensions{.eps,.pdf,.png,.jpg}
\else
  \DeclareGraphicsExtensions{.eps}
\fi

\newcommand{\norm}[1]{\left\lVert#1\right\rVert}

\def\R{\mathbb{R}}

\renewenvironment{equation*}{\[}{\]\ignorespacesafterend}
\usepackage[normalem]{ulem}
\setstcolor{red}

\begin{document}

\title{On the hardness of the $L_1-L_2$ regularization problem
}

\author{Yuyuan Ouyang         \and
        Kyle Yates
}

\institute{Y. Ouyang \at
              School of Mathematical and Statistical Sciences, Clemson University \\
              \email{yuyuano@clemson.edu}           
           \and
           K. Yates \at
              School of Mathematical and Statistical Sciences, Clemson University\\
              \email{kjyates@clemson.edu}  
}

\date{Received: date / Accepted: date}

\maketitle

\begin{abstract}
The sparse linear reconstruction problem is a core problem in signal processing which aims to recover sparse solutions to linear systems. The original problem regularized by the total number of nonzero components (also known as $L_0$ regularization) is well-known to be NP-hard. The relaxation of the $L_0$ regularization by using the $L_1$ norm offers a convex reformulation, but is only exact under certain conditions (e.g., restricted isometry property) which might be NP-hard to verify. 
To overcome the computational hardness of the $L_0$ regularization problem while providing tighter results than the $L_1$ relaxation, several alternate optimization problems have been proposed to find sparse solutions. One such problem is the $L_1-L_2$ minimization problem, which is to minimize the difference of the  $L_1$ and $L_2$ norms subject to linear constraints. This paper proves that solving the $L_1-L_2$ minimization problem is NP-hard. Specifically, we prove that it is NP-hard to minimize the $L_1-L_2$ regularization function subject to linear constraints. Moreover, it is also NP-hard to solve the unconstrained formulation that minimizes the sum of a least squares term and the $L_1-L_2$ regularization function. Furthermore, restricting the feasible set to a smaller one by adding nonnegative constraints does not change the NP-hardness nature of the problems.



\end{abstract}

\section{Introduction}

The sparse linear reconstruction problem, also known as the compressed sensing problem (CS), is a core problem in signal processing. The (CS) problem is the problem of finding a sparsest solution to a linear system. The original (CS) problem is the optimization problem
\begin{equation}\label{eq:CS}
        \tag{CS}
\begin{split}
    \min_{\textbf{x}\in \R^n} & \norm{\textbf{x}}_0\\
    \text{subject to } & \textbf{Ax}=\textbf{b},
\end{split}
\end{equation}
where $\norm{\textbf{x}}_{0}$ denotes the number of nonzero entries of $\textbf{x}$, also known as the $L_0$ norm. Of particular interest for several optimization problems such as problem \eqref{eq:CS} is their computational hardness. In computational complexity theory, a common characterization on the hardness of a problem is its NP-hardness. Specifically, a decision problem is in the nondeterministic polynomial-time (NP) class if it can be solved by a nondeterministic Turing machine in polynomial time, and it is NP-hard if every problem in NP can be reduced to it by a polynomial-time reduction. Informally, a problem is NP if it is easy to verify whether a proposed solution is correct, and a problem is NP-hard if solving it is as hard as solving any NP problem. For a detailed description of NP, NP-hardness, and other topics in computational complexity theory, see, e.g, \cite{garey1979computers}. The common understanding is that a problem is considered difficult to solve if it can be proven to be NP-hard.

Unfortunately, problem (\ref{eq:CS}) is known to be NP-hard \cite{Natarajan1995SparseAS} (See also \cite{foucart2013compressive}). Specifically, one can construct special instances of problem (\ref{eq:CS}) such that solving them would yield solutions to the exact 3-cover problem, which is already known to be NP-hard. Consequently, problem (\ref{eq:CS}) must also be NP-hard. Informally, here the key idea is to demonstrate that problem (\ref{eq:CS}) is at least as hard as a known NP-hard problem, hence establishing its NP-hardness. This analysis strategy, known as polynomial-time reduction, is a standard technique for proving NP-hardness in computational complexity theory.

Several alternate optimization problems have been proposed to approximate solutions to problem (\ref{eq:CS}) in order to overcome this computational hardness. This includes techniques such as minimizing the $L_1$ norm subject to $\textbf{Ax}=\textbf{b}$ for instance, which can lead to exact recovery of a sparsest solution under certain conditions \cite{doi:10.1073/pnas.0437847100,candes2005stablesignalrecoveryincomplete}. However, verifying the exact recovery conditions of the $L_1$ relaxation could itself be NP-hard; see, e.g., \cite{bandeira2013certifying,tillmann2013computational} on the NP-hardness of certifying the restricted isometry property and the nullspace property. Several other models have been proposed in the hope of providing tighter relaxation than the $L_1$ model, e.g., $L_p$ \cite{frank1993statistical}, smoothly clipped absolute deviation (SCAD) \cite{fan2001variable}, minimax  concave penalty (MCP) \cite{zhang2010nearly}, etc. Unfortunately, many such models have too been shown to be NP-hard (see, e.g., \cite{Chen2011ComplexityOU,JMLR:v20:17-373,10.5555/2110269.2110276,doi:10.1080/00949650902773544,articlecomplexity,Nguyen2019NPhardnessO} and the references within). In spite of the unfavorable complexity results, it should be noted that several numerical methods have been developed that can efficiently solve these problems in practice (see, e.g., \cite{10.5555/3600270.3603093,Tibshirani_2011}).

Among the regularization strategies in sparse solution recovery, one important strategy is the $L_1-L_2$ minimization problem \cite{doi:10.1137/13090540X}. The idea is to add a penalty term that is the difference of the $L_1$ and $L_2$, yielding a non-convex optimization problem
\begin{equation}\label{eqnmain}
\begin{split}
    \min_{\textbf{x}\in \R^n}& \norm{\textbf{x}}_1 - \norm{\textbf{x}}_2\\
    \textup{subject to }& \textbf{Ax}=\textbf{b}. \\
    \end{split}
\end{equation}
For a fixed penalty parameter $\lambda >0$, the unconstrained version of problem (\ref{eqnmain}) is the regularization model
\begin{equation}\label{unconstrainedequation}
    \begin{split}
    \min_{\textbf{x}\in \R^{n}} &\norm{\textbf{Ax}-\textbf{b}}_2^2 + \lambda(\norm{\textbf{x}}_1 - \norm{\textbf{x}}_2). 
    \end{split}
 \end{equation}
We will refer to problems (\ref{eqnmain}) and (\ref{unconstrainedequation}) as the constrained (CP) and unconstrained $L_1-L_2$ minimization problems (UP), respectively. Various studies have been conducted on constrained and unconstrained $L_1-L_2$ minimization (see, e.g., \cite{BI2022337,bui2021weighteddifferenceanisotropicisotropic,Li20201,Yin2015MinimizationO}, and the references within), but most focus on the effectiveness of recovering sparse signals.

To the best of our knowledge, there has been no results developed in the literature concerning the computational complexity of  problems (\ref{eqnmain}) and (\ref{unconstrainedequation}). What makes the $L_1-L_2$ model interesting from the complexity analysis point of view is that one could not derive any NP-hardness result based on previously developed hard instances of other non-convex regularization models. Indeed, for all the papers on the NP-hardness results of (CS)
\cite{Chen2011ComplexityOU,JMLR:v20:17-373,10.5555/2110269.2110276,doi:10.1080/00949650902773544,articlecomplexity,Nguyen2019NPhardnessO}, none are applicable to problems  (\ref{eqnmain}) and (\ref{unconstrainedequation}). Among the NP-hardness proofs, the most general of such results is probably \cite{JMLR:v20:17-373}, in which they prove the strong NP-hardness for several unconstrained problems with a broad class of  loss and penalty functions. This class consists of several sparse optimization approximators but does not include problem (\ref{unconstrainedequation}). In particular, \cite{JMLR:v20:17-373} shows that for given $\textbf{A}=(\textbf{a}_1,\dots,\textbf{a}_m)^\top \in \R^{m\times n}$ and any parameter $\lambda >0$, the problem
\begin{align*}
	\begin{split}
		\min_{\textbf{x}\in \R^n} &\sum_{i=1}^{m} \ell(\textbf{a}_i^\top \textbf{x}, b_i) + \lambda \sum_{j=1}^{n} p (|x_j|)
	\end{split}
\end{align*}
is strongly NP-hard for some  loss functions $\ell: \R\times \R \rightarrow \R^+$ and some penalty functions $p: \R \rightarrow \R^+$, where $\R^+$ denotes the nonnegative real numbers. This result does not apply to our objective function in problem (\ref{unconstrainedequation}) however, as the term $(\norm{\textbf{x}}_1 - \norm{\textbf{x}}_2)$ can not be written in the penalty form required by \cite{JMLR:v20:17-373}. Thus, the computational hardness of problem (\ref{unconstrainedequation}) had therefore remained open.

It should be noted that extensive efforts on solving problems \eqref{eqnmain} and \eqref{unconstrainedequation} have been reported in the literature, despite the unknown status of their computational hardness. Observing that problems \eqref{eqnmain} and \eqref{unconstrainedequation} belong to the class of \textit{difference of convex} (DC) programming problems,
in \cite{Yin2015MinimizationO} the authors
develop an iterative method based on a DC algorithm (DCA) to numerically solve the unconstrained problem \eqref{unconstrainedequation}. The findings in \cite{Yin2015MinimizationO} suggest that $L_1-L_2$ minimization via DCA outperforms $L_p$ solvers in recovering sparse signals for ill-conditioned matrices. The authors also reported that $L_1-L_2$ minimization via DCA outperforms the alternating direction method of multipliers (ADMM) \cite{8186925} for $L_1$ regularization in every case. Subsequent studies (\cite{10.1007/978-3-319-18161-5_15,10.1007/s10915-014-9930-1}) back up these findings for both the constrained and unconstrained versions of $L_1-L_2$ minimization, suggesting that these models may offer a promising solution to finding sparse solutions in practice. Other follow-up studies also studies the methods for improving the performance of the ADMM method for $L_1-L_2$ minimization problems (see, e.g., \cite{lou2018fast}).

\subsection{Our Contributions} In this paper, we prove the NP-hardness of both the constrained and unconstrained $L_1-L_2$ minimization problems. We prove (\ref{eqnmain}) is NP-hard by showing a polynomial-time reduction from the NP-complete partition problem to a variation of constrained $L_1-L_2$ minimization with an additional fixed positive parameter, which holds for decision variables in both $\R^+$ and $\R$. We then prove that (\ref{unconstrainedequation}) is NP-hard for a range of $\lambda$, also by providing a polynomial-time reduction from the partition problem to the unconstrained $L_1-L_2$ minimization problem. Much like the constrained version, our NP-hardness results again hold for decision variables in both $\R^+$ and $\R$.

\subsection{Notation}
For any positive integer $n$, we denote by $[n]$ the set of natural numbers from $1$ to $n$, $\textbf{I}_n$ the $n\times n $ identity matrix, and $\textbf{1}_n\in\mathbb{R}^n$ the vector of $n$ ones. We use bold lowercase lettering for vectors and bold uppercase lettering for matrices.

\section{NP-Hardness of Constrained $L_1-L_2$ Minimization}\label{constrained.section}

In this section, we prove that it is NP-hard to solve the constrained $L_1-L_2$ minimization problem (\ref{eqnmain}). In fact, we will prove a slightly more general statement that for any parameter
 $\tau\in (0,\sqrt{2})$, it is NP-hard to solve the following constrained minimization problem:
\begin{equation}
	\label{eq:L1tauL2}
        \tag{CP}
	\begin{split}
	\min_{\textbf{x}\in \R^n}& \norm{\textbf{x}}_1 - \tau\norm{\textbf{x}}_2\\
	\textup{subject to }& \textbf{Ax}=\textbf{b}. \\
	\end{split}
\end{equation}
Throughout this paper, we call the above the constrained $L_1-\tau L_2$ problem (\ref{eq:L1tauL2}). While \eqref{eq:L1tauL2} reduces to problem (\ref{eqnmain}) when $\tau=1$, our NP-hardness result on more general $\tau$ in an interval shows that the major computational hardness of solving the $L_1-L_2$ problem is the nonconvex $(-\|\textbf{x}\|_2)$ term; such hardness could not be alleviated by setting a slightly smaller coefficient parameter $\tau$ in front of the nonconvex $(-\|\textbf{x}\|_2)$ term. As a byproduct of our analysis, we will also show that the $L_1-L_2$ problem does not become any easier when we only search its solution from a smaller feasible set by enforcing additional nonnegative constraints. Specifically, for any $\tau>0$, we will show that the following problem is also NP-hard:
\begin{equation}
	\label{eq:L1tauL2_nonneg}
        \tag{NCP}
	\begin{split}
	\min_{\textbf{x}\in \R^n}& \norm{\textbf{x}}_1 - \tau\norm{\textbf{x}}_2\\
	\textup{subject to }& \textbf{Ax}=\textbf{b} \text{ and }\textbf{x}\ge \textbf{0}.
	\end{split}
\end{equation}
We call the above the nonnegative constrained $L_1-\tau L_2$ problem (\ref{eq:L1tauL2_nonneg}).

Our strategy of proving the NP-hardness of \eqref{eq:L1tauL2} and \eqref{eq:L1tauL2_nonneg} is a polynomial-time reduction from the classical partition problem, which is known to be NP-hard (see, e.g., \cite{garey1979computers}). 
Given a multiset $S=\{a_1,\ldots,a_m\}$ of integers or rational numbers, the 
\emph{partition problem} is to decide whether there is a partition of $S$ into two subsets $S_1$ and $S_2$ such that the sum of elements in $S_1$ is equal to that in $S_2$. In the following proposition, we describe the partition problem as an optimization problem.

\begin{proposition}
	\label{thm:partition_opt}
	The problem of determining the solvability of the partition problem concerning multiset $S=\{a_1,\ldots,a_m\}$ is equivalent to determining whether the optimal objective value of the following optimization problem is $-m$:
	\begin{equation}
		\label{eq:opt_for_partition}
		\begin{split}
			\min_{\textbf{u},\textbf{v}\in \R^{m}} & \sum_{i=1}^m -u_i^2-v_i^2\\
			\text{subject to} & 
			\begin{aligned}[t]\quad u_i+v_i= &1 \quad\text{ for } i=1,\dots,m
				\\
				\textup{\textbf{a}}^\top (\textup{\textbf{u}} - \textup{\textbf{v}})=& 0
                    \\
                    \textup{\textbf{u}},\textup{\textbf{v}} \ge &  \textup{\textbf{0}
                    }.
			\end{aligned}		
		\end{split}
	\end{equation}	
    Here we denote $\textup{\textbf{a}}:=(a_1,\ldots,a_m)^\top$.
\end{proposition}

\begin{proof}
    Consider a relaxation modification of problem \eqref{eq:opt_for_partition} in which the constraint $\textbf{a}^\top (\textbf{u}-\textbf{v})=0$ is removed:
	\begin{equation}
		\label{eq:opt_for_partition_relaxed}
		\begin{split}
			\min_{\textbf{u},\textbf{v}\in \R^{m}} & \sum_{i=1}^m -u_i^2-v_i^2\\
			\text{subject to} & 
			\begin{aligned}[t]\quad u_i+v_i= &1 \quad\text{ for } i=1,\dots,m
				\\
                    \textbf{u},\textbf{v} \ge &   \textbf{0}.
			\end{aligned}		
		\end{split}
	\end{equation}	
    The optimal value of 
    problem \eqref{eq:opt_for_partition_relaxed} is $-m$ and its set of optimal solutions is
    \begin{equation}
        \label{eq:Xstar_NCP}
        X^*:=\{(\textbf{u}^*,\textbf{v}^*)\in\mathbb{R}^m\times \mathbb{R}^m\vert (u^*_i,v^*_i) = (1,0) \text{ or }(0,1)\text{ for all }i\in [m]\}.
    \end{equation}
    To see this, note that for any feasible solution to 
    problem \eqref{eq:opt_for_partition_relaxed} we have $-(u_i^2+v_i^2)\ge -(u_i+v_i)^2= 1$ due to nonnegativity constraints, in which the inequality becomes equality if and only if $u_iv_i=0$. Since the relaxation is done by removing a constraint in the original problem \eqref{eq:opt_for_partition}, we can also observe that the optimal value of the original problem \eqref{eq:opt_for_partition} is at least $-m$.

    Suppose that there exists a solution to the partition problem with respect to multiset $S = \{a_1,\ldots,a_m\}$. Let us introduce vectors $\textbf{u},\textbf{v}\in\mathbb{R}^m$, in which $u_i=1$ (or $v_i=1$ respectively) if $a_i$ is partitioned into the first subset (or second subset respectively), and $u_i=0$ (or $v_i=0$ respectively) otherwise. Clearly, we have $(\textbf{u},\textbf{v})\in X^*$. Moreover, since the sum of elements in the two subsets are the same, we have $\textbf{a}^\top \textbf{u} = \textbf{a}^\top \textbf{v}$. Therefore $(\textbf{u},\textbf{v})$ is an optimal solution to the relaxed problem \eqref{eq:opt_for_partition_relaxed}, while also feasible for the original problem \eqref{eq:opt_for_partition}. Consequently, $(\textbf{u},\textbf{v})$ must be an optimal solution to the original problem.

    Conversely, for a given partition problem on multiset $S = \{a_1,\ldots,a_m\}$, suppose that $(\textbf{u}^*,\textbf{v}^*)$ is an optimal solution to problem \eqref{eq:opt_for_partition} with optimal value $-m$. Then $(\textbf{u}^*,\textbf{v}^*)$ is also clearly an optimal solution to the relaxed problem \eqref{eq:opt_for_partition_relaxed} and hence $(\textbf{u}^*,\textbf{v}^*)\in X^*$. By partitioning element $a_i$ to the first set if $u_i^*=1$ and to the second set if $v_i^*=1$ for all $i\in[m]$, we obtain a solution to the partition problem. \hfill $\qed$
\end{proof}

The optimization formulation of the partition problem in problem \eqref{eq:opt_for_partition} is not necessarily novel; similar formulations have been studied in computational complexity theory. 
From the complexity analysis point of view, the NP-hardness of the partition problem potentially arises from its discrete nature as described in the proof of Proposition \ref{thm:partition_opt}, in which any optimal solution in the set (\ref{eq:Xstar_NCP}) has binary component values. From the sparse reconstruction point of view, the 
result of Proposition \ref{thm:partition_opt} states that the partition problem is equivalent to a linear reconstruction with nonnegativity constraints and a negative squared Euclidean norm ($-\|\cdot\|_2^2$) regularization. Such linear reconstruction problem yields $m$-sparse solutions as shown in (\ref{eq:Xstar_NCP}): any optimal solution has exactly $m$ non-zero components and $m$ zero components.
Here the nonnegativity constraints are important in enforcing the optimal solutions to the $m$-sparse solution set (\ref{eq:Xstar_NCP}).

From the optimization description of the partition problem in Proposition \ref{thm:partition_opt}, we can prove immediately that solving problem \eqref{eq:L1tauL2_nonneg} is NP-hard.

\begin{theorem}
    \label{thm:NCP}
    For any $\tau>0$, it is NP-hard to solve the nonnegative constrained $L_1-\tau L_2$ problem \eqref{eq:L1tauL2_nonneg}.
\end{theorem}
\begin{proof}
    Fixing any $\tau>0$, we will show that there is a polynomial-time reduction from the partition problem to problem \eqref{eq:L1tauL2_nonneg}. Specifically, supposing that we have an instance of the partition problem with multiset $S=\{a_1,\dots,a_m\}$, consider the optimization problem \eqref{eq:L1tauL2_nonneg} in which $n=2m$, $\textbf{x} = (u_1,\ldots,u_m,v_1,\ldots,v_m)^\top$, and 
	$$
	\textbf{A} = \begin{bmatrix}
		\textbf{I}_m & \textbf{I}_m \\
		\textbf{a}^\top & -\textbf{a}^\top
	\end{bmatrix},
	\quad \quad
	\textbf{b}=
	\begin{bmatrix}
		\textbf{1}_m\\
		0
	\end{bmatrix},
	$$
	where $\textbf{I}_m$ is an $m\times m $ identity matrix, $\textbf{1}_m$ is a vector of $m$ ones, and $\textbf{a} := (a_1,\dots,a_m)^\top$. Noting that the constraints $\textbf{Ax}=\textbf{b}$, $\textbf{x}\ge\textbf{0}$ implies that $\textbf{u}+\textbf{v}=\textbf{1}_m$, $\textbf{u},\textbf{v}\ge \textbf{0}$ and hence $\|\textbf{x}\|_1$ becomes a fixed constant $m$, we can observe that the aforementioned instance of problem \eqref{eq:L1tauL2_nonneg} is equivalent to the problem 
    	\begin{align*}
			\min_{\textbf{u},\textbf{v}\in \R^{m}} & m - \sqrt{\sum_{i=1}^m  u_i^2 + v_i^2}\\
			\text{subject to} & 
			\begin{aligned}[t]\quad u_i+v_i= &1 \quad\text{ for } i=1,\dots,m
				\\
				\textbf{a}^\top (\textbf{u} - \textbf{v})=& 0
                    \\
                    \textbf{u},\textbf{v} \ge &  \textbf{0},
			\end{aligned}		
\end{align*}
 which has an identical solution set to problem
    \eqref{eq:opt_for_partition}.
    By Proposition \ref{thm:partition_opt}, solving such \eqref{eq:L1tauL2_nonneg} instance is equivalent to determining the solvability of the partition problem, which is NP-hard. \hfill $\qed$
\end{proof}

From the sparse reconstruction point of view, under the nonnegativity constraint both the negative Euclidean norm and $L_1-\tau L_2$ minimization are able to capture the sparse optimal solutions of the partition problem instances. However, such sparsity capturing is intrinsically related to the NP-hardness of the partition problem, yielding the NP-hardness of the $L_1-\tau L_2$ problem. 

In the optimization formulation of the partition problem, the nonnegativity constraint plays an important role in capturing the sparse optimal solutions. In the sequel, we will prove that the $L_1-\tau L_2$ minimization is able to capture the sparse solutions of the partition problem without requiring the nonnegativity constraints. Specifically, we will consider the following instance of the \eqref{eq:L1tauL2} problem:
\begin{equation}
	\label{eq:L1tauL2_for_partition}
	\begin{split}
		\min_{\textbf{u},\textbf{v}\in \R^{m}} & \sum_{i=1}^m (|u_i|+|v_i|) - \tau\cdot \sqrt{\sum_{i=1}^m u_i^2+v_i^2}\\
		\text{subject to} & 
		\begin{aligned}[t]\quad u_i+v_i= &1 \quad\text{ for } i=1,\dots,m
		\\
		\textbf{a}^\top (\textbf{u} - \textbf{v})=& 0.
	\end{aligned}		
	\end{split}
\end{equation}
Given an instance of the partition problem with multiset $S=\{a_1,\dots,a_m\}$, observe that problem \eqref{eq:L1tauL2_for_partition} is simply the specific case of problem \eqref{eq:L1tauL2} in which $n=2m$, $\textbf{x} = (u_1,\ldots,u_m,v_1,\ldots,v_m)^\top$, and
$$
\textbf{A} = \begin{bmatrix}
		\textbf{I}_m & \textbf{I}_m \\
		\textbf{a}^T & -\textbf{a}^T
	\end{bmatrix},
	\quad \quad
	\textbf{b}=
	\begin{bmatrix}
		\textbf{1}_m\\
		0
	\end{bmatrix}.$$ We will show that solving 
    problem \eqref{eq:L1tauL2_for_partition} is equivalent to solving the partition problem and hence solving the \eqref{eq:L1tauL2} problem is NP-hard.
Similar to the strategy utilized in the proof of Proposition \ref{thm:partition_opt}, we will study the relaxed version
\begin{equation}
	\label{eq:L1tauL2_for_partition_relaxed}
	\begin{split}
		\min_{\textbf{u},\textbf{v}\in \R^{m}} & \sum_{i=1}^m (|u_i|+|v_i|) - \tau\cdot \sqrt{\sum_{i=1}^m u_i^2+v_i^2}\\
		\text{subject to} & 
		\begin{aligned}[t]\quad u_i+v_i= &1 \quad\text{ for } i=1,\dots,m
	\end{aligned}		
	\end{split}
\end{equation}
and show that its set of optimal solutions is $X^*$ defined in \eqref{eq:Xstar_NCP}.

However, unlike the proof of Proposition \ref{thm:partition_opt} in which the set of optimal solutions $X^*$ can be derived easily, since we no longer have the nonnegativity constraints, some additional arguments are necessary to derive the set of optimal solutions to the 
relaxed problem \eqref{eq:L1tauL2_for_partition_relaxed}. The analysis is summarized in the following lemma.

\begin{lemma}
    \label{thm:CP_opt}
    For any fixed $\tau \in (0,\sqrt{2})$, the set of optimal solutions to problem \eqref{eq:L1tauL2_for_partition_relaxed} is $X^*$ defined in \eqref{eq:Xstar_NCP}.
\end{lemma}
\begin{proof}
    We make the following claim: for any integer
    $k=0,\dots,m$, there always exist optimal solutions to the problem
    \begin{equation}
	\label{eq:L1tauL2_for_partition_relaxed_analysis}
	\begin{split}
		\min_{\textbf{u},\textbf{v}\in \R^{m}} & \sum_{i=1}^m (|u_i|+|v_i|) - \tau\cdot \sqrt{\sum_{i=1}^m u_i^2+v_i^2}\\
		\text{subject to} & 
		\begin{aligned}[t]\quad & u_i+v_i= 1 \quad\text{ for } i=1,\dots,m,
            \\
            & \textbf{u}\ge \textbf{0},\ v_1,\ldots,v_k\le 0,\ v_{k+1},\ldots,v_{m}\ge 0.
	\end{aligned}		
	\end{split}
    \end{equation}
    Note when $k=0$ or $k=m$, we interpret equation \eqref{eq:L1tauL2_for_partition_relaxed_analysis} as the case with $\textbf{v}\geq \textbf{0}$ or $\textbf{v}\leq \textbf{0}$, respectively.
    For any feasible solution to \eqref{eq:L1tauL2_for_partition_relaxed} for each $i$ either both $u_i$ and $v_i$ are nonnegative, or exactly one is positive and one is negative since $u_i + v_i = 1$. Without loss of generality, assume for every $i$ that $u_i$ is nonnegative. Thus we may instead study \eqref{eq:L1tauL2_for_partition_relaxed_analysis} for every integer $k=0,\dots, m$ in place of \eqref{eq:L1tauL2_for_partition_relaxed}. Moreover, any optimal solution  $(\textbf{u}^*,\textbf{v}^*)$ to 
    \eqref{eq:L1tauL2_for_partition_relaxed_analysis} must satisfy $v_{1}^*=\cdots=v_k^*=0$. Based on the claim, it is straightforward to observe that the optimization problem \eqref{eq:L1tauL2_for_partition_relaxed} is equivalent to the following one with additional nonnegative constraints: 
\begin{equation}\label{eq:lemma1_eqn_nonnegative_constrained}
        \begin{split}
		\min_{\textbf{u},\textbf{v}\in \R^{m}} & \sum_{i=1}^m (|u_i|+|v_i|) - \tau\cdot \sqrt{\sum_{i=1}^m u_i^2+v_i^2}\\
		\text{subject to} & 
		\begin{aligned}[t]\quad & u_i+v_i= 1 \quad\text{ for } i=1,\dots,m
            \\
            & \textbf{u}, \textbf{v}\ge 0.
	\end{aligned}		
	\end{split}
    \end{equation}
We conclude the proof by observing that the  
nonnegative constrained problem \eqref{eq:lemma1_eqn_nonnegative_constrained} is equivalent to the optimization problem \eqref{eq:opt_for_partition_relaxed}, whose set of optimal solutions is $X^*$.

    To finish the proof it suffices to prove the claim. Observe in problem \eqref{eq:L1tauL2_for_partition_relaxed_analysis} that for any feasible solution $(\textbf{u},\textbf{v})$,
    \begin{align*}
    \begin{split}
    & \sum_{i=1}^m (|u_i|+|v_i|) - \tau\cdot \sqrt{\sum_{i=1}^m u_i^2+v_i^2} 
    \\
     \geq & \sum_{i=1}^k (1-v_i - v_i) + \sum_{i=k+1}^{m} (u_i+v_i) - \tau\cdot \sqrt{\sum_{i=1}^k ((1-v_i)^2+v_i^2) + \sum_{i=k+1}^m (u_i+v_i)^2}
    \\
    = & m - 2\sum_{i=1}^{k}v_i - \tau \sqrt{m-k + \sum_{i=1}^{k} ((1-v_i)^2+v_i^2)},
    \end{split}
    \end{align*}
    in which the inequality becomes equality if and only if $(u_i,v_i)=(0,1)$ or $(1,0)$ for all 
    $i=k+1,\dots,m$. Letting $v_i=-w_i^2$ for all $i=1,\ldots,k$, based on the above observation it suffices to study the problem $\min_{\textbf{w}\in\R^k} g(\textbf{w})$, in which
    \begin{align*}
    \begin{split}
        g(\textbf{w}):=& 2\sum_{i=1}^{k}w_i^2 - \tau \sqrt{m-k + \sum_{i=1}^{k} ((1+w_i^2)^2+w_i^4)}
        \\
        = & 2\|\textbf{w}\|_2^2 - \tau\sqrt{m+2\|\textbf{w}\|_2^2 + 2\|\textbf{w}\|_4^4}.
    \end{split}
    \end{align*}
    In fact, to prove the claim it suffices to show that the function $g(\textbf{w})$ has a unique minimizer $\textbf{w}^*=0$.

    We will first show that the continuous function $g(\textbf{w})$ has a minimizer by proving that it is a coercive function. See, e.g., Corollary 2.5 in \cite{guler2010foundations} for the relationship between coercivity and existence of global optimal solutions. Observe that
    \begin{align*}
        g(\textbf{w}) \ge & 2\|\textbf{w}\|_2^2 - \tau\left(\sqrt{m} + \sqrt{2}\|\textbf{w}\|_2 + \sqrt{2}\sqrt{\sum_{i=1}^m w_i^4}\right)
        \\
        \ge & 2\|\textbf{w}\|_2^2 - \tau\left(\sqrt{m} + \sqrt{2}\|\textbf{w}\|_2 + \sqrt{2}\sum_{i=1}^m w_i^2\right)
        \\
        = & (2-\sqrt{2}\tau)\|\textbf{w}\|_2^2 - \sqrt{m}\tau - \sqrt{2}\tau\|\textbf{w}\|_2.
    \end{align*}
Since $\tau<\sqrt{2}$, for any sequence $\{\textbf{w}^{(j)}\}_{j=1}^\infty$ such that $\|\textbf{w}^{(j)}\|_2\to \infty$, we have $g(\textbf{w}^{(j)})\to \infty$ also. Hence $g$ is a coercive function. Consequently, $g$ is also continuous the problem $\min_{\textbf{w}}g(\textbf{w})$ has at least one global optimizer. 

Since the function $g$ is differentiable, we may compute its minimizers by studying its stationary points that solve the system $\nabla g(\textbf{w}) = 0$. Here the $i$-th equation is 
$$
    4w_i - \tau\cdot \frac{4w_i + 8w_i^3}{2\sqrt{m + 2\|\textbf{w}\|_2^2 +  2\|\textbf{w}\|_4^4}} = 0,
$$
which is equivalent to either $w_i=0$ or
$$
     \frac{2}{\tau} \sqrt{m + 2\|\textbf{w}\|_2^2 +  2\|\textbf{w}\|_4^4} - 1 - 2w_i^2=0.
$$
However, for any
$\tau \in (0,\sqrt{2})$, the latter has no feasible solution since
\begin{align*}
    \frac{2}{\tau} \sqrt{m + 2\|\textbf{w}\|_2^2 +  2\|\textbf{w}\|_4^4} &\ge \sqrt{2}\cdot\sqrt{1 + 2w_i^2 + 2w_i^4} \\
    &> \sqrt{1 + 4w_i^2 + 4w_i^4} \\
    &= 1+2w_i^2.
\end{align*}
So the only stationary point is $\textbf{w}=0$ which is the global optimizer. Our claim is now proved. \hfill $\qed$
\end{proof}

With the help of Lemma \ref{thm:CP_opt}, we are now ready to prove that solving problem \eqref{eq:L1tauL2} is NP-hard by following the same argument in the NP-hardness proof of \eqref{eq:L1tauL2_nonneg}.

\begin{theorem}
    \label{thm:CP}
    For any 
    $\tau \in (0,\sqrt{2})$, it is NP-hard to solve the constrained $L_1-\tau L_2$ problem \eqref{eq:L1tauL2}.
\end{theorem}
\begin{proof}
    Since the set of optimal solutions to problem \eqref{eq:L1tauL2_for_partition_relaxed} is also $X^*$ defined in \eqref{eq:Xstar_NCP}, by the same argument as in the proof of Proposition \ref{thm:partition_opt}, we can prove that the problem of determining the solvability of the partition problem concerning multiset $S=\{a_1,\ldots,a_m\}$ is equivalent to determining whether the optimal objective value of problem \eqref{eq:L1tauL2_for_partition} is $m-\tau\sqrt
    m$. Here the value $m-\tau\sqrt
    m$ is the optimal objective value of problem \eqref{eq:L1tauL2_for_partition_relaxed} in the optimal solution set $X^*$. Moreover, note that problem \eqref{eq:L1tauL2_for_partition} is an instance of \eqref{eq:L1tauL2} with $n=2m$, $\textbf{x} = (u_1,\ldots,u_m,v_1,\ldots,v_m)^\top$, and 
	$$
	\textbf{A} = \begin{bmatrix}
		\textbf{I}_m & \textbf{I}_m \\
		\textbf{a}^\top & -\textbf{a}^\top
	\end{bmatrix},
	\quad \quad
	\textbf{b}=
	\begin{bmatrix}
		\textbf{1}_m\\
		0
	\end{bmatrix}.
	$$
    Therefore, solving such \eqref{eq:L1tauL2} instance is equivalent to determining the solvability of the partition problem, which is NP-hard. \hfill $\qed$
\end{proof}

Our constructed hard instance \eqref{eq:L1tauL2_for_partition} suggests that the constrained $L_1-\tau L_2$ minimization captures the sparse solutions of the partition problem and hence reflects its NP-hardness. An  immediate research question is whether we can reformulate the linear reconstruction problem to an unconstrained version instead and control the reconstruction sparsity by tuning a parameter in the regularization path. However, as we will show in the next section, the unconstrained $L_1-L_2$ minimization problem will also capture the sparse solutions of the partition problem.

\section{NP-Hardness of Unconstrained $L_1-L_2$ Minimization}
\label{unconstrained.section}

In this section, we prove that it is also NP-hard to solve the unconstrained $L_1-L_2$ minimization problem (\ref{unconstrainedequation}). Recall the unconstrained $L_1-L_2$ minimization problem (\ref{unconstrainedequation}) is the problem
\begin{equation}\label{eq:UP}\tag{UP}
    \begin{split}
    \min_{\textbf{x}\in \R^{n}} &\norm{\textbf{Ax}-\textbf{b}}_2^2 + \lambda(\norm{\textbf{x}}_1 - \norm{\textbf{x}}_2) 
    \end{split}
 \end{equation}
 for fixed penalty parameter $
 \lambda > 0$. Throughout this section we will refer to the above as the unconstrained $L_1-L_2$ problem (\ref{eq:UP}). We will also use (\ref{eq:NUP}) to denote the nonnegativity constrained version of problem (\ref{eq:UP}):
\begin{equation}\label{eq:NUP}\tag{NUP}
    \begin{split}
    \min_{\textbf{x}\in \R^{n}} & \norm{\textbf{Ax}-\textbf{b}}_2^2 + \lambda(\norm{\textbf{x}}_1 - \norm{\textbf{x}}_2) \\
    &\text{subject to } \textbf{x}\geq \textbf{0}.
    \end{split}
 \end{equation}
 By utilizing the linear reconstruction instance \eqref{eq:L1tauL2_for_partition}, we will show in this section that both (\ref{eq:UP}) and (\ref{eq:NUP}) are NP-hard.
 Similar to our analysis in the constrained case (\ref{eq:L1tauL2}) and (\ref{eq:L1tauL2_nonneg}), our analysis is based on showing that the $L_1-L_2$ regularization captures the sparse solutions of the partition problem and hence reflects its NP-hardness.

Specifically, we will consider \eqref{eq:UP} and \eqref{eq:NUP} instances
	\begin{equation}
		\label{eq:opt_for_partition3}
		\begin{split}
			  \min_{\textbf{u},\textbf{v}\in \R^m}  f(\textbf{u},\textbf{v}):=& (\textbf{a}^\top(\textbf{u}-\textbf{v}))^2+\sum_{i=1}^{m} (u_i+v_i -1)^2 
     \\
     & + \lambda\left[\sum_{i=1}^{m}|u_i|+|v_i| - \sqrt{\sum_{i=1}^m u_i^2 + v_i^2}\right],
		\end{split}
	\end{equation}	

\begin{equation}
    \label{eq:opt_for_partition2}
    \begin{split}
          \min_{\textbf{u},\textbf{v}\in \R^m} f(\textbf{u},\textbf{v}):=& (\textbf{a}^\top(\textbf{u}-\textbf{v}))^2+\sum_{i=1}^{m} (u_i+v_i -1)^2 
          \\
          & + \lambda\left[\sum_{i=1}^{m}u_i+v_i - \sqrt{\sum_{i=1}^m u_i^2 + v_i^2}\right]\\
  \text{subject to }  & \textbf{u},\textbf{v}\ge \textbf{0}.
    \end{split}
\end{equation}	
We will show that for any fixed regularization parameter $\lambda\in (0,2)$, solving either one of the instances (\ref{eq:opt_for_partition3}) or  (\ref{eq:opt_for_partition2}) is NP-hard.

Before conducting our complexity analysis, we will first need to make sure that both problems (\ref{eq:opt_for_partition3}) and (\ref{eq:opt_for_partition2}) are solvable. The existence of their optimal solutions are shown in Proposition \ref{prop31}.

\begin{proposition}\label{prop31}
    For any regularization parameter $\lambda>0$, there exists at least one optimal solution to
	problem
     (\ref{eq:opt_for_partition3}).
\end{proposition}

\begin{proof}
    We claim that if 
    $\max\{|u_j|,|v_j|\}>1+\sqrt{m}+2m/\lambda$ for any index $j$, then $f(\textbf{u},\textbf{v})>f(\textbf{0},\textbf{0})$. By our claim, we conclude that it suffices to study the minimization of $f(u,v)$ within the compact set $\|(\textbf{u},\textbf{v})\|_\infty \le 1+\sqrt{m}+2m/\lambda$. By Weierstrass theorem, the optimal solution of such problem must exist.

    To prove the claim, note that if 
     $\max\{|u_j|,|v_j|\}>1+\sqrt{m}+2m/\lambda$ for some index $j$, then there are exactly four possible cases:
\begin{enumerate}[label=(\arabic*)]
    \item $|u_j|>1+\sqrt{m}+2m/\lambda$ and $|v_j|\le 2m/\lambda$;
    \item $|v_j|>1+\sqrt{m}+2m/\lambda$ and $|u_j|\le 2m/\lambda$;
    \item $|u_j|>1+\sqrt{m}+2m/\lambda$ and $|v_j|> 2m/\lambda$;
    \item $|v_j|>1+\sqrt{m}+2m/\lambda$ and $|u_j|> 2m/\lambda$.
\end{enumerate}
    In the first case, we have
$
    |u_j + v_j - 1| \ge |u_j-1| - |v_j| \ge |u_j| - 1 - |v_j|>\sqrt{m}.
$
Noting that the $L_1-L_2$ regularization function is nonnegative, we have $f(\textbf{u},\textbf{v})\ge (u_j+v_j-1)^2>m = f(\textbf{0},\textbf{0})$.
A similar result also holds for the second case. 
In the third case, we have 
$|u_j|> 1 +\sqrt{m} + 2m/\lambda >2m/\lambda$ and $|v_j|>2m/\lambda$. Focusing on the regularization part we have 
\begin{align*}
\begin{split}
    f(\textbf{u},\textbf{v}) \ge & \lambda \left[\sum_{i=1}^{m}|u_i|+|v_i| - \sqrt{\left( \sqrt{u_j^2 + v_j^2} + \sum_{i=1,i\not=j}^m |u_i| + |v_i|\right)^2}\right] 
    \\
    = & \lambda(|u_j| + |v_j| - \sqrt{u_j^2 + v_j^2}).
\end{split}
\end{align*}
Observing that
\begin{align*}
&(|u_j| + |v_j| - \sqrt{u_j^2 + v_j^2})\cdot \frac{|u_j| + |v_j| + \sqrt{u_j^2 + v_j^2}}{|u_j| + |v_j| + \sqrt{u_j^2 + v_j^2}} = \frac{(|u_j| + |v_j|)^2 - u_j^2 - v_j^2}{|u_j| + |v_j| + \sqrt{u_j^2 + v_j^2}}
\\
=& \frac{2|u_j||v_j|}{|u_j| + |v_j| + \sqrt{u_j^2 + v_j^2}}
\end{align*}
and
$$
|u_j| + |v_j| + \sqrt{u_j^2 + v_j^2} \leq 2(|u_j| + |v_j|),
$$
we then have
\begin{align*}
f(\textbf{u},\textbf{v}) \geq & \frac{2\lambda |u_j||v_j|}{|u_j| + |v_j| + \sqrt{u_j^2 + v_j^2}}
    \ge  \frac{\lambda|u_j||v_j|}{|u_j|+|v_j|} 
    \\
    \ge &  \frac{\lambda\min\{|u_j|,|v_j|\}\max\{|u_j|,|v_j|\}}{2\max\{|u_j|,|v_j|\}}
    \\
    \ge & \frac{\lambda}{2}\min\{|u_j|,|v_j|\} >  m = f(\textbf{0},\textbf{0}).
\end{align*}
In the fourth case, we have $|u_j| >2m/\lambda$ and $|v_j|>1 +\sqrt{m} + 2m/\lambda > 2m/\lambda$, in which the argument is identical to the third case.
\hfill $\qed$
\end{proof}

We remark that the proof for existence of optimal solutions to problem (\ref{eq:opt_for_partition3}) extends straightforwardly to the version of the problem with nonnegative variables. This is since our proof strategy is based on restricting the optimal solutions to a compact infinity-norm ball, which remains compact when intersected with nonnegativity constraints. 
We formally state this result 
as Corollary \ref{existence.corollary}.
\begin{corollary}\label{existence.corollary}
    There exists at least one optimal solution to problem (\ref{eq:opt_for_partition2}).
\end{corollary}
\begin{proof}
    Immediate from the proof of Proposition \ref{prop31} and restricting $\textbf{u},\textbf{v}\ge \textbf{0}$. \hfill $\qed$
\end{proof}

We are now ready to prove the NP-hardness of the
(\ref{eq:UP}) and (\ref{eq:NUP}) instances in equations \eqref{eq:opt_for_partition3} and \eqref{eq:opt_for_partition2} (respectively) in the upcoming sections. We will start with the (\ref{eq:NUP}) instance \eqref{eq:opt_for_partition2} and derive a polynomial reduction from the partition problem.

\subsection{NP-Hardness of the (\ref{eq:NUP})  instance when $\lambda\in(0,2)$}

In this section, we will show that the (\ref{eq:NUP}) instance \eqref{eq:opt_for_partition2} can be reduced from the partition problem for all $\lambda\in (0,2)$. Assuming that the partition problem with respect to multiset $S=\{a_1,\ldots,a_m\}$ has a solution, we will consider the following lower bound problem:
\begin{equation}\label{eq:nunL1-L2}
\begin{split}
    \min_{\textbf{u},\textbf{v}\in \R^m} & g(\textbf{u},\textbf{v}) = \sum_{i=1}^{m} (u_i+v_i -1)^2 + \lambda\left[\sum_{i=1}^{m}u_i+v_i - \sqrt{\sum_{i=1}^m u_i^2 + v_i^2}\right]\\
    & \text{subject to } u_i,v_i \geq 0 \quad \text{for } i=1,\dots,m.
\end{split}
\end{equation}
Here $g(\textbf{u},\textbf{v})\le f(\textbf{u},\textbf{v})$ is a lower bound function with one less fitting term $(\textbf{a}^\top(\textbf{u}-\textbf{v}))^2$ concerning the partition requirement. When $\lambda\in (0,2)$, the set of optimal solutions to this lower bound problem is stated in Lemma \ref{lemma.optset}.

\begin{lemma}\label{lemma.optset}
    Suppose that $\lambda\in (0,2)$. Let $c(\lambda): =\frac{1}{2}\left(\lambda/\sqrt{m}-\lambda + 2\right)$ and 
    $$
    g^*(\lambda) := \lambda \left( (1-\frac{\lambda}{4})\cdot m + (\frac{\lambda}{2} - 1)\cdot \sqrt{m} - \frac{\lambda}{4}\right).
    $$
    The optimal objective value to problem (\ref{eq:nunL1-L2}) is $g^*(\lambda)$ with optimal solution set
    $$
    Y^* = \{ (\textup{\textbf{u}}^*,\textup{\textbf{v}}^*) \in \R^m \times \R^m | (u_i^*,v_i^*) = (c(\lambda),0) \textup{ or } (0,c(\lambda)) \textup{ for all }i\in [m]\}.
    $$
\end{lemma}

\begin{proof}
For the objective function $g(\textbf{u},\textbf{v})$ in problem (\ref{eq:nunL1-L2}), note that the nonnegativity of any feasible solution $(\textbf{u},\textbf{v})$ yields
\begin{align*}
    g(\textbf{u},\textbf{v}) \ge \sum_{i=1}^{m} (u_i+v_i -1)^2 + \lambda \left( \sum_{i=1}^{m}u_i+v_i - \sqrt{\sum_{i=1}^m (u_i + v_i)^2}\right).
\end{align*}
Here the equality holds if and only if $u_iv_i=0$ for all $i$. Therefore, it suffices to study the optimization problem
\begin{equation}\label{eq:alternate}
\begin{split}
    \min_{\textbf{t}\in\R^m} \  h(\textbf{t}):=&\sum_{i=1}^{m} (t_i-1)^2 + \lambda \left(\sum_{i=1}^{m} t_i - \sqrt{\sum_{i=1}^m t_i^2}\right)\\
    & \text{subject to } \textbf{t}\ge \textbf{0}.
\end{split}
\end{equation}
Note that for all $\textbf{t} = \textbf{u}+\textbf{v}$, we have $g(\textbf{u},\textbf{v})\ge h(\textbf{t})$. Moreover, for
any optimizer $\textbf{t}^*$ of  
problem \eqref{eq:alternate}, we can recover an optimizer $(\textbf{u}^*,\textbf{v}^*)$ of problem (\ref{eq:nunL1-L2}) in which either $(u_i^*,v_i^*)=(0,t_i^*)$ or $(u_i^*,v_i^*)=(t_i^*,0)$ for all $i$. Furthermore, we claim that every optimal solution to \eqref{eq:nunL1-L2} is of this form. Observe that for any feasible solution $(\textbf{u},\textbf{v})$ to \eqref{eq:nunL1-L2} with $u_iv_i > 0$ for some $i$, the solution $\textbf{t}$ with $t_i = u_i + v_i$ for all $i$ is feasible to \eqref{eq:alternate} and satisfies $h(\textbf{t}) < g(\textbf{u},\textbf{v})$. We can then construct the solution $(\tilde{\textbf{u}},\tilde{\textbf{v}}) = (\textbf{t},0)$ to \eqref{eq:nunL1-L2} which satisfies $g(\tilde{\textbf{u}},\tilde{\textbf{v}}) = h(\textbf{t}) < g(\textbf{u},\textbf{v})$. Hence, any feasible $(\textbf{u},\textbf{v})$ to \eqref{eq:nunL1-L2} with $u_iv_i > 0$ for some $i$ is never optimal, and the set of optimal solutions to \eqref{eq:nunL1-L2} is completely determined by the set of optimal solutions to \eqref{eq:alternate}.

By Corollary \ref{existence.corollary}, optimal solutions to problem (\ref{eq:nunL1-L2}) exist, so optimal solutions to problem (\ref{eq:alternate}) must exist too. Note that $\textbf{t}=\textbf{0}$ is the only point of $h:\R^n \rightarrow \R$ which is not differentiable. Except the
point $\textbf{t}=\textbf{0}$, with nonnegativity constraints the Karush–Kuhn–Tucker (KKT) condition is necessary for optimality, i.e., if $\textbf{t}$ is an optimal solution, then there exists multipliers $\mu_i$'s such that
\begin{align}
    \label{eq:KKT_NUP}
    \mu_i\ge 0,\ 2(t_i-1) +\lambda - \frac{\lambda t_i}
    {\sqrt{\sum_{j=1}^m t_j^2}} - \mu_i = 0,\ (-t_i)\mu_i &= 0,\ \forall i\in[m].
\end{align}

We now show that there exists a unique KKT point with smaller objective function value than the non-differentiable point $\textbf{t}=\textbf{0}$. If there exists $i\in [m]$ such that the multiplier $\mu_i>0$, then by \eqref{eq:KKT_NUP} we have $t_i=0$ and $
-2 + \lambda - \mu_i =0$, contradicting our assumption that $\lambda<2$. Therefore, any KKT point must have multipliers $\mu_i =0$ for all $i\in [m]$. By \eqref{eq:KKT_NUP}, for all $i,k\in [m]$ we have 
\begin{equation}\label{eq:relation_for_KKT}
 \lambda t_i  - 2(t_i-1)\sqrt{\sum_{j=1}^m t_j^2} = \lambda\sqrt{\sum_{j=1}^m t_j^2} = \lambda t_k - 2(t_k-1)\sqrt{\sum_{j=1}^m t_j^2},
\end{equation}
The relation in \eqref{eq:relation_for_KKT} yields $t_i = t_k$ since $\lambda t_i - 2(t_i - 1) = \lambda t_k - (2t_k - 1)$, as $\sqrt{\sum_{j=1}^m t_j^2 }\neq 0$. Thus we have
$$
\lambda t_i  - 2(t_i-1)\sqrt{m}t_i = \lambda\sqrt{m}t_i,
$$
whose nonnegative solution is 
$t_i = c(\lambda)$. Thus, there is a unique KKT point $\textbf{t}^*$ with $t_i^*=c(\lambda)$ for all $i\in[m]$. Since $\lambda\in (0,2)$, it is then easy to verify that 
\begin{align*}
    h(\textbf{t}^*) = g^*(\lambda) < \lambda(1-\lambda/4)m < m = h(\textbf{0}).
\end{align*}
Therefore, $\textbf{t}^*$ is the unique optimal solution to problem \eqref{eq:alternate}. \hfill $\qed$
\end{proof}

The optimal set resulting from 
Lemma \ref{lemma.optset} will be enough for us to solve the partition problem by solving a similarly constructed optimization problem.

\begin{proposition}
	\label{thm:partition_opt2}
	Suppose $\lambda\in (0,2)$. The problem of determining the solvability of the partition problem concerning multiset $S=\{a_1,\ldots,a_m\}$ is equivalent to determining whether the optimal objective value of the following optimization problem is $g^*(\lambda)$:
	\begin{align*}
		\begin{split}
			  \min_{\textup{\textbf{u}},\textup{\textbf{v}}\in \R^m} \  (\textup{\textbf{a}}^\top(\textup{\textbf{u}}-\textup{\textbf{v}}))^2+\sum_{i=1}^{m} (u_i+v_i &-1)^2 + \lambda\left[\sum_{i=1}^{m}u_i+v_i - \sqrt{\sum_{i=1}^m u_i^2 + v_i^2}\right]\\
     \text{subject to }  & \ \textup{\textbf{u}},\textup{\textbf{v}}\ge \textup{\textbf{0}}.
		\end{split}
	\end{align*}	
    Here we denote $\textup{\textbf{a}}:=(a_1,\ldots,a_m)^\top$.
\end{proposition}

\begin{proof}
    Let $c(\lambda),g^*(\lambda)$, and $Y^*$ be as defined in Lemma \ref{lemma.optset}. It is clear that for any feasible $u,v$ to problem (\ref{eq:opt_for_partition2}) we have
    $$
    (\textbf{a}^\top(\textbf{u}-\textbf{v}))^2+\sum_{i=1}^{m} (u_i+v_i -1)^2 + \lambda\left[\sum_{i=1}^{m}u_i+v_i - \sqrt{\sum_{i=1}^m u_i^2 + v_i^2}\right] \ge g^*(\lambda),
    $$
    with equality holding if and only if $(\textbf{u},\textbf{v})\in Y^*$ and $\textbf{a}^\top (\textbf{u}-\textbf{v}) =\textbf{0}$.

    Suppose that there exists a solution to the partition problem with respect to multiset $S = \{a_1,\ldots,a_m\}$. Let us introduce vectors $\textbf{u},\textbf{v}\in\mathbb{R}^m$, in which $u_i=c(\lambda)$ (or $v_i=c(\lambda)$ respectively) if $a_i$ is partitioned into the first subset (or second subset respectively), and $u_i=0$ (or $v_i=0$ respectively) otherwise. Clearly, we have $(\textbf{u},\textbf{v})\in Y^*$. Moreover, since the sum of elements in the two subsets are the same, we have $\textbf{a}^\top \textbf{u} = \textbf{a}^\top \textbf{v}$. Therefore, problem (\ref{eq:opt_for_partition2}) achieves an optimal objective value of $g^*(\lambda)$.

    Conversely, for a given partition problem on multiset $S = \{a_1,\ldots,a_m\}$, suppose that $(\textbf{u}^*,\textbf{v}^*)$ is an optimal solution to problem (\ref{eq:opt_for_partition2}) with optimal value $g^*(\lambda)$. Then, we must have $\textbf{a}^\top (\textbf{u}^*-\textbf{v}^*)=0$ and  $(\textbf{u}^*,\textbf{v}^*)$ as an optimal solution to problem (\ref{eq:nunL1-L2}). Hence $(\textbf{u}^*,\textbf{v}^*)\in Y^*$.  By partitioning element $a_i$ to the first set if $u_i^*=c(\lambda)$ and to the second set if $v_i^*=c(\lambda)$ for all $i\in[m]$, we obtain a solution to the partition problem. \hfill $\qed$
\end{proof}

We are now ready to prove that problem \eqref{eq:NUP} is NP-hard. Similar to our proof of the constrained version, we can find a specific instance of problem \eqref{eq:NUP} which is equivalent to problem \eqref{eq:opt_for_partition2}, thus solving the partition problem. Theorem \ref{thm:NP-hard-NUP} formally states the result.

\begin{theorem}\label{thm:NP-hard-NUP}
    For any $\lambda\in (0,2)$, it is NP-hard to solve the nonnegative unconstrained $L_1-L_2$ problem \eqref{eq:NUP}.
\end{theorem}

\begin{proof}
    Fixing any $\lambda\in (0,2)$, we will show that there is a polynomial-time reduction from the partition problem to problem \eqref{eq:NUP}. Specifically, supposing that we have an instance of the partition problem with multiset $S=\{a_1,\dots,a_m\}$, consider the optimization problem \eqref{eq:NUP} in which $n=2m$, $\textbf{x} = (u_1,\ldots,u_m,v_1,\ldots,v_m)^\top$, and 
	$$
	\textbf{A} = \begin{bmatrix}
		\textbf{I}_m & \textbf{I}_m \\
		\textbf{a}^\top & -\textbf{a}^\top
	\end{bmatrix},
	\quad \quad
	\textbf{b}=
	\begin{bmatrix}
		\textbf{1}_m\\
		0
	\end{bmatrix},
	$$
	where $\textbf{I}_m$ is an $m\times m $ identity matrix and $\textbf{1}_m$ is a vector of $m$ ones. By construction, it is clear that this instance of problem \eqref{eq:NUP} is equivalent to problem \eqref{eq:opt_for_partition2}. By Proposition \ref{thm:partition_opt2}, solving such \eqref{eq:NUP} instance is equivalent to determining the solvability of the partition problem, which is NP-hard. \hfill $\qed$
\end{proof}

\subsection{NP-Hardness of the \eqref{eq:UP} instance when $\lambda\in (0,2)$}

Having finished proving the NP-hardness of problem \eqref{eq:NUP}, we are now ready to move onto problem \eqref{eq:UP}. Consider the optimization problem
\begin{equation}\label{eq:unL1-L2}
	\begin{split}
		\min_{\textbf{u},\textbf{v}\in \R^m} g(\textbf{u},\textbf{v}) &= \sum_{i=1}^{m} (u_i+v_i -1)^2 + \lambda\left[\sum_{i=1}^{m}|u_i|+|v_i| - \sqrt{\sum_{i=1}^m u_i^2 + v_i^2}\right]
	\end{split}.
\end{equation}
Our strategy will be to prove the optimal set of \eqref{eq:unL1-L2} is again $Y^*$. As we have already established the existence of optimal solutions to problem \eqref{eq:unL1-L2} in Proposition \ref{prop31}, it suffices to use first order optimality conditions to find any possible optimal solutions, which is the strategy we use
in Lemma \ref{lemma.optset2}. 

\begin{lemma}\label{lemma.optset2}
    Suppose that $\lambda\in (0,2)$. Let $c(\lambda) =\frac{1}{2}\left(\lambda/\sqrt{m}-\lambda + 2\right)$ and 
    $$
    g^*(\lambda) = \lambda \left( (1-\frac{\lambda}{4})\cdot m + (\frac{\lambda}{2} - 1)\cdot \sqrt{m} - \frac{\lambda}{4}\right).
    $$
    The optimal objective value to problem \eqref{eq:unL1-L2} is $g^*(\lambda)$ with optimal solution set
    $$
    Y^* = \{ (\textup{\textbf{u}}^*,\textup{\textbf{v}}^*) \in \R^m \times \R^m | (u_i^*,v_i^*) = (c(\lambda),0) \textup{ or } (0,c(\lambda)) \textup{ for all }i\in [m]\}.
    $$
\end{lemma}

\begin{proof}
We make a few observations on points $(\textbf{u},\textbf{v})$ that would not be optimal for the optimization problem \eqref{eq:unL1-L2}. First, if $(\textbf{u},\textbf{v})$ satisfies $u_i+v_i<0$ for some $i$, then there exists $(\tilde{\textbf{u}},\tilde{\textbf{v}})$ such that $g(\textbf{u},\textbf{v})>g(\tilde{\textbf{u}},\tilde{\textbf{v}})$. In fact, if
\begin{align*}
    \tilde{u}_j = \begin{cases}
        -u_i & j=i
        \\
        u_j & j\not =i,
    \end{cases} \quad
    \tilde{v}_j = \begin{cases}
        -v_i & j=i
        \\
        v_j & j\not =i,
    \end{cases}
\end{align*}
then
\begin{align*}
    g(\textbf{u},\textbf{v}) - g(\tilde{\textbf{u}},\tilde{\textbf{v}}) = (u_i+v_i-1)^2 - (u_i+v_i+1)^2 = (-2)\cdot (2u_i+2v_i) > 0.
\end{align*}

Second, if $(\textbf{u},\textbf{v})$ satisfies $u_i=v_i=0$ for some $i$, then there exists $(\tilde{\textbf{u}},\tilde{\textbf{v}})$ such that $g(\textbf{u},\textbf{v})>g(\tilde{\textbf{u}},\tilde{\textbf{v}})$. In fact, if
\begin{align*}
    \tilde{u}_j = \begin{cases}
        1-\frac{1}{2}\lambda & j=i
        \\
        u_j & j\not =i,
    \end{cases}
    \tilde{v}_j = \begin{cases}
        0 & j=i
        \\
        v_j & j\not =i,
    \end{cases}
\end{align*}
then letting $C:=\sum_{j=1,j\not=i}^{m} u_j^2 + v_j^2$ we have 
\begin{align*}
\begin{split}
    g\textbf{(u},\textbf{v}) - g(\tilde{\textbf{u}},\tilde{\textbf{v}}) = & [1 - \lambda\sqrt{C}] - \left[\frac{1}{4}\lambda^2 + \lambda\left(1-\frac{1}{2}\lambda-\sqrt{(1-\frac{1}{2}\lambda)^2+C}\right)\right]
    \\
    > & 1 - \frac{1}{4}\lambda^2 - \lambda(1-  \frac{1}{2}\lambda)
    \\
    = & \frac{1}{4}(2 - \lambda)^2
    \\
    \ge & 0.
\end{split}
\end{align*}

Third, if $(\textbf{u},\textbf{v})$ satisfies $u_i>0$ and $v_i>0$ for some $i$, then there exists $(\tilde{\textbf{u}},\tilde{\textbf{v}})$ such that $g(\textbf{u},\textbf{v})>g(\tilde{\textbf{u}},\tilde{\textbf{v}})$. In fact, if
\begin{align*}
    \tilde{u}_j = \begin{cases}
        u_i+v_i & j=i
        \\
        u_j & j\not =i,
    \end{cases}
    \tilde{v}_j = \begin{cases}
        0 & j=i
        \\
        v_j & j\not =i,
    \end{cases}
\end{align*}
then letting $C:=\sum_{j=1,j\not=i}^{m} u_j^2 + v_j^2$ we have
\begin{align*}
\begin{split}
    g(\textbf{u},\textbf{v}) - g(\tilde{\textbf{u}},\tilde{\textbf{v}}) = & \lambda\sqrt{(u_i+v_i)^2 + C} -\lambda\sqrt{u_i^2+v_i^2 + C}>0.
\end{split}
\end{align*}
 
By the above observations, for any optimal solution $(\textbf{u},\textbf{v})$ to problem \eqref{eq:unL1-L2}  we may partition its indices by the following:
\begin{align*}
    \begin{split}
        I_1:=\{i|u_i>0,v_i=0\}, I_2:=\{i|u_i>0, v_i<0\}.
    \end{split}
\end{align*}

Since each index of our optimal solution $(\textbf{u},\textbf{v})$ of consideration belongs now to either $I_1$ and $I_2$, from first-order necessary optimality conditions we have
\begin{align*}
    \begin{split}
        0 = & \frac{\partial g}{\partial u_i}(\textbf{u},\textbf{v}) = 2(u_i-1) + \lambda - \frac{\lambda u_i}{\sqrt{\sum_{j=1}^m u_j^2+v_j^2}},\ \forall i\in I_1,
        \\
        0 = & \frac{\partial g}{\partial u_i}(\textbf{u},\textbf{v}) = 2(u_i+v_i - 1) + \lambda - \frac{\lambda u_i}{\sqrt{\sum_{j=1}^m u_j^2+v_j^2}},\ \forall i\in I_2,
        \\
        0 = & \frac{\partial g}{\partial v_i}(\textbf{u},\textbf{v}) = 2(u_i+v_i - 1) - \lambda - \frac{\lambda v_i}{\sqrt{\sum_{j=1}^m u_j^2+v_j^2}},\ \forall i\in I_2,
    \end{split}
\end{align*}
which yield (after performing subtraction and addition on the last two relations)
\begin{equation}
		\label{eq:tmp}
    \begin{split}
        &  \left( 2 - \frac{\lambda}{\sqrt{\sum_{j=1}^m u_j^2+v_j^2}}\right) u_i = 2 - \lambda,\ \forall i\in I_1,
        \\
        & \frac{\lambda }{\sqrt{\sum_{j=1}^m u_j^2+v_j^2}}(u_i-v_i) = 2\lambda,\ \forall i\in I_2,
        \\
        & \left(4 - \frac{\lambda}{\sqrt{\sum_{j=1}^m u_j^2+v_j^2}}\right)(u_i+v_i) = 4,\ \forall i\in I_2.
    \end{split}
\end{equation}
From the 
relations in \eqref{eq:tmp} we have $u_j=u_k$ for all $j,k\in I_1$ and $u_j=u_k$ and $v_j=v_k$ for all $j,k\in I_2$. Let us denote $u_j=t$ for all $j\in I_1$ and $u_j=r$, $v_j=s$ for all $j\in I_2$. Applying our notations and noting that  $\lambda>0$, from the second equation  
in \eqref{eq:tmp} we have
\begin{equation}\label{eq:I1I2_relation}
    \sqrt{|I_1|t^2 + |I_2|(r^2+s^2)} =  \frac{1}{2}(r-s)
    .
\end{equation}
Note that the  
relation in \eqref{eq:I1I2_relation} implies that $I_2=\emptyset$. To see this, assume by contrary that $|I_2|\ge 1$. We would then have
\begin{align*}
    \frac{1}{4}(r-s)^2 \le \frac{1}{2}(r^2+s^2) < |I_2|(r^2+s^2) \le  |I_1|t^2 + |I_2|(r^2+s^2),
\end{align*}
yielding a contradiction. Hence we could only have $I_2=\emptyset$. Recalling that $u_j=t$ and $v_j=0$ for all $j\in I_1$, from the first equation of 	\eqref{eq:tmp}
we can solve that
$t = \frac{1}{2}\left(\lambda/\sqrt{m}-\lambda + 2\right)$. \hfill $\qed$    
\end{proof}

The rest of our analysis is then similar to the previous section concerning the hard \eqref{eq:NUP} instance: the optimal set $Y^*$ is enough to find an optimization problem which solves the partition problem. We state the rest in Proposition \ref{thm:partition_opt3}.

\begin{proposition}
	\label{thm:partition_opt3}
	Suppose that $\lambda\in (0,2)$. The problem of determining the solvability of the partition problem concerning multiset $S=\{a_1,\ldots,a_m\}$ is equivalent to determining whether the optimal objective value of the following optimization problem is $g^*(\lambda)$:
    \begin{align}\label{eqn:partition_opt3}
		\begin{split}
			  \min_{\textup{\textbf{u}},\textup{\textbf{v}}\in \R^m}  &(\textup{\textbf{a}}^\top(\textup{\textbf{u}}-\textup{\textbf{v}}))^2+\sum_{i=1}^{m} (u_i+v_i -1)^2 \\
              &+ \lambda\left[\sum_{i=1}^{m}|u_i|+|v_i| - \sqrt{\sum_{i=1}^m u_i^2 + v_i^2}\right].
		\end{split}
	\end{align}	
    Here we denote $\textup{\textbf{a}}:=(a_1,\ldots,a_m)^\top$.
\end{proposition}

\begin{proof}
    Let $c(\lambda),g^*(\lambda)$, and $Y^*$ be as defined in Lemma \ref{lemma.optset2}. It is clear that for any $\textbf{u},\textbf{v}$ we have
    $$
    (\textbf{a}^\top(\textbf{u}-\textbf{v}))^2+\sum_{i=1}^{m} (u_i+v_i -1)^2 + \lambda\left[\sum_{i=1}^{m}|u_i|+|v_i| - \sqrt{\sum_{i=1}^m u_i^2 + v_i^2}\right] \ge g^*(\lambda),
    $$
    with equality holding if and only if $(\textbf{u},\textbf{v})\in Y^*$ and $\textbf{a}^\top (\textbf{u}-\textbf{v}) =0$. The remainder of the proof follows identically the proof of Proposition \ref{thm:partition_opt2} for problem \eqref{eq:unL1-L2}. \hfill $\qed$
\end{proof}

We are now ready to prove our main result that solving  \eqref{eq:UP} is NP-hard.

\begin{theorem}\label{thm.NP-hard-up}
    For any $\lambda\in (0,2)$, it is NP-hard to solve the unconstrained $L_1-L_2$ problem \eqref{eq:UP}.
\end{theorem}

\begin{proof}
    Fixing any $\lambda\in (0,2)$, we will show that there is a polynomial-time reduction from the partition problem to problem \eqref{eq:UP}. Specifically, supposing that we have an instance of the partition problem with multiset $S=\{a_1,\dots,a_m\}$, consider the optimization problem \eqref{eq:UP} in which $n=2m$, $\textbf{x} = (u_1,\ldots,u_m,v_1,\ldots,v_m)^\top$, and 
	$$
	\textbf{A} = \begin{bmatrix}
		\textbf{I}_m & \textbf{I}_m \\
		\textbf{a}^\top & -\textbf{a}^\top
	\end{bmatrix},
	\quad \quad
	\textbf{b}=
	\begin{bmatrix}
		\textbf{1}_m\\
		0
	\end{bmatrix},
	$$
	where $\textbf{I}_m$ is an $m\times m $ identity matrix and $\textbf{1}_m$ is a vector of $m$ ones. By construction, it is clear that this instance of problem \eqref{eq:UP} is equivalent to problem 
     \eqref{eqn:partition_opt3}. By Proposition \ref{thm:partition_opt3}, solving such \eqref{eq:UP} instance is equivalent to determining the solvability of the partition problem, which is NP-hard. \hfill $\qed$
\end{proof}

\subsection{NP-hardness when $\lambda \geq 2$} 
While our NP-hardness proofs in the previous two sections assume that $\lambda\in(0,2)$, we will show in this section that when $\lambda \geq 2$, problems \eqref{eq:UP} and \eqref{eq:NUP} are still NP-hard. Specifically, let us consider problem instances in which
$$
\textbf{A} = \sqrt{\lambda} \cdot \begin{bmatrix}
		 \textbf{I}_m & \textbf{I}_m \\
		\textbf{a}^\top & -\textbf{a}^\top
	\end{bmatrix},
	\quad \quad
	\textbf{b}=
	\sqrt{\lambda}\cdot \begin{bmatrix}
		 \textbf{1}_m\\
		0
	\end{bmatrix},
$$
and $\textbf{x} = (u_1,\dots, u_m, v_1,\dots, v_m)$. Observe that the instance of \eqref{eq:UP} with this choice of $\textbf{A}$ and $\textbf{b}$ 
is then the optimization problem
	\begin{align*}
		\begin{split}
			  \min_{\textbf{u},\textbf{v}\in \R^m} \  (\sqrt{\lambda} \textbf{a}^\top(\textbf{u}&-\textbf{v}))^2  +\sum_{i=1}^{m} (\sqrt{\lambda} u_i+\sqrt{\lambda} v_i -\sqrt{\lambda})^2 \\ & + \lambda\left[\sum_{i=1}^{m}u_i+v_i - \sqrt{\sum_{i=1}^m u_i^2 + v_i^2}\right].
		\end{split}
	\end{align*}
    Clearly, we may factor $\lambda$ and instead consider the optimization problem
    	\begin{align*}
		\begin{split}
			   \min_{\textbf{u},\textbf{v}\in \R^m} \  ( \textbf{a}^\top(\textbf{u}-\textbf{v}))^2+\sum_{i=1}^{m} ( u_i+ v_i &-1)^2 + \sum_{i=1}^{m}u_i+v_i - \sqrt{\sum_{i=1}^m u_i^2 + v_i^2},
		\end{split}
	\end{align*}
which is NP-hard by the proof of Theorem \ref{thm.NP-hard-up} (with $\lambda=1$). The same argument may be made for problem \eqref{eq:NUP} with the proof of Theorem \ref{thm:NP-hard-NUP}, meaning both \eqref{eq:UP} and \eqref{eq:NUP} are actually NP-hard for any finite $\lambda >0$.

There is still a slight distinction between the cases of $\lambda \in (0,2)$ and $\lambda \geq 2$ in our NP-hardness results. When $\lambda \geq 2$, the above choice of $\textbf{A}$ and $\textbf{b}$ must depend on the value of $\lambda$ to prove NP-hardness with our proof techniques. On the contrary, When $\lambda \in (0,2)$, the choice of $\textbf{A}$ and $\textbf{b}$ stated in Theorems \ref{thm:NP-hard-NUP} and \ref{thm.NP-hard-up} does not depend on $\lambda$. Considering the fact that $\lambda$ is usually selected as a tuning parameter while $\textbf{A}$ and $\textbf{b}$ are already determined, our results on the case with $\lambda\in (0,2)$ are slightly stronger.

\section{Concluding Remarks} In this paper, we prove that it is NP-hard to solve both the constrained and unconstrained formulation of the linear reconstruction problem subject to the $L_1-L_2$ minimization. We also show that restricting the feasible set to a smaller one by adding nonnegative constraints does not change the NP-hardness nature of the problems. 

While our results are negative on the computational tractability of the $L_1-L_2$ problem, our NP-hardness proof is based on a worse-case instance. It is interesting to study whether there are a smaller class of $L_1-L_2$ problems whose exact solutions could be solved in polynomial time. 

It should also be noted that our result focuses on the computational tractability issue of the exact optimal solution. While our result reveals that solving the exact solution is NP-hard, it remains an interesting problem on studying the computational tractability of an approximate solution. Recently, there has been new research results developed on the NP-hardness of solving approximate solutions to certain sparse optimization problems; see \cite{JMLR:v20:17-373} and the references within. While the $L_1-L_2$ regularized problem does not belong to the problems covered in \cite{JMLR:v20:17-373}, it is interesting to see if the reduction techniques in \cite{JMLR:v20:17-373} could shed some light on the complexity analysis of approximation of the $L_1-L_2$ regularized problem.
 
\section*{Conflict of interest}
 The authors declare that they have no conflict of interest.

\section*{Declarations}
\textbf{Funding:} Part of this research was supported by NSF award DMS-1913006.

\section*{Data Availability Statement}

The authors do not analyze or generate any datasets because this work is primarily a theoretical mathematical approach.

\bibliographystyle{spmpsci}
\bibliography{bibliography}

\end{document}